\documentclass[letterpaper,two column, 10 pt, conference]{ieeeconf}  

\IEEEoverridecommandlockouts                              
\usepackage[utf8]{inputenc}
\usepackage{amsmath,bm}
\usepackage{amsfonts}
\usepackage{amssymb}
\usepackage{graphicx}
\usepackage{amsfonts}
\usepackage[colorlinks,urlcolor=blue]{hyperref}
\usepackage{caption}
\usepackage{subcaption}
\usepackage{algorithm}
\usepackage{algpseudocode}
\usepackage{cases}
\usepackage{enumerate}
\usepackage{xcolor}
\usepackage{cite}
\usepackage{verbatim}
\usepackage[normalem]{ulem}

\DeclareMathOperator{\blkdiag}{\mathrm{blkdiag}}
\DeclareMathOperator{\col}{\mathrm{col}}

\newtheorem{Theorem}{Theorem}

\newtheorem{Lemma}{Lemma}
\newtheorem{Problem}{Problem}
\newtheorem{Corollary}{Corollary}
\newtheorem{Remark}{Remark}
\newtheorem{Assumption}{Assumption}
\usepackage{enumerate}
\usepackage{amsmath,bm}
\usepackage{multirow}
\newcommand{\m}[1]{\mathbf{#1}}
\newcommand{\mc}[1]{\mathcal{#1}}
\newcommand{\mb}[1]{\mathbb{#1}}
\newcommand{\abs}[1]{\lVert{#1} \rVert}

\newcommand\numberthis{\addtocounter{equation}{1}\tag{\theequation}}
\usepackage{tikz}
	\usetikzlibrary{arrows,positioning,bending,shapes,backgrounds,fit}  
	\tikzstyle{frame} = [draw, -latex]
	\tikzstyle{line} = [draw]
	\tikzstyle{line2} = [draw, dashdotted]
	\tikzstyle{line3} = [draw, dashed]
	\tikzstyle{line3UD} = [draw, dashed]
	\tikzstyle{place} = [circle, draw=black, fill=white, thick, inner sep=2pt, minimum size=1mm]
	\tikzstyle{place2} = [circle, draw=black, fill=black, thick, inner sep=2pt, minimum size=1mm]
	\tikzstyle{placeRed} = [circle, draw=red, fill=red, thick, inner sep=2pt, minimum size=1mm]
	\tikzstyle{vertex} = [circle, draw=black, fill=black, thick, inner sep=2pt, minimum size=1mm]
\usepackage{tkz-euclide}
\makeatletter
\def\endthebibliography{%
  \def\@noitemerr{\@latex@warning{Empty `thebibliography' environment}}%
  \endlist
}
\makeatother

\title{\LARGE Bearing-constrained Formation Tracking Control of Nonholonomic Agents without Inter-agent Communication}

\author{Quoc Van Tran and Jinwhan Kim
\thanks{The authors are with the Department of Mechanical Engineering, Korea Advanced Institute of Science and Technology (KAIST), Daejeon, Republic of Korea. Emails: $\{$quoctran; jinwhan$\}$@kaist.ac.kr}
\thanks{This work is supported by the BK21 FOUR Program of the National Research Foundation Korea (NRF) grant funded by the Ministry of Education (MOE), in part by the Future Mobility Testbed Development through IT, AI, and Robotics.}
}
\begin{document}
\maketitle
\begin{abstract}
This letter presents two bearing-constrained formation tracking control protocols for multiple nonholonomic agents based respectively on the bearing vectors and displacements between the agents. The desired formation pattern of the system is specified by the desired inter-agent bearing vectors. In the proposed control schemes, there are two or more leaders moving with the same constant velocity; the other follower agents do not have the information of the leaders' velocity nor communicate variables with their neighbors. Under both the proposed control laws, the system achieves the moving target formation asymptotically. Simulation results are provided to support the theoretical development.
\end{abstract}
\section{Introduction}
Bearing-constrained formation control, in which multiple agents form the desired formation specified by the desired inter-agent bearings, has attracted much interest from researchers in recent years  \cite{Zhao2016tac,Quoc2018ccta, Schiano2016,Quoc2018tcns,Xli2020TCyber}. If the control protocol utilizes merely the inter-agent bearing vectors, it is referred to as a bearing-only control law \cite{Bishop2011, Zhao2019tac}. Bearing measurements can be obtained by passive sensors such as optical cameras \cite{Tron2016csm} or wireless sensor arrays. Thus, bearing-only control approaches may be favored in military and marine applications where signal transmission is restricted.

Formation tracking requires the agents to both form the desired formation and move with the same reference velocity simultaneously, for which the leader-follower approach has often been used. Based on the inter-agent displacements, \cite{XPeng2020Tmecha} {addressed} formation tracking for nonholonomic agents with a directed tree graph and the knowledge of the neighbor's velocity of each agent. Distance-based formation tracking control laws were proposed for systems of single-integrator \cite{Qingkai2018SCL} and nonholonomic agents \cite{Khaledyan2018acc} by either utilizing consensus-based (centroid) estimators or assigning the reference velocity to all agents. 
Bearing-based formation tracking using the inter-agent displacements and with two or more leaders has been investigated for single-integrator \cite{Zhao2017}, double-integrator \cite{Zhao2017}, and unmanned autonomous vehicles \cite{YHuang2021tcns}. The control law for double-integrator agents in \cite{Zhao2017} requires in addition the relative velocities between the agents. \cite{YHuang2021tcns} employs consensus-based observers to estimate the leaders' velocity.  

Formation tracking based merely on the inter-agent bearings is a challenging problem because bearing-only control laws are always bounded even when the position errors grow unbounded \cite{Zhao2019tac}. Furthermore, the formation's scale is not explicitly observable from bearing-only measurements \cite{Schiano2016}. In order to fix the formation's scale, bearing-only formation tracking control methods often employ at least two leaders moving with the same reference velocity.
Bearing-only formation tracking laws using constant-velocity leaders have been proposed for systems with single-integrator \cite{Zhao2019tac}, double-integrator \cite{Zhao2019tac,Minh2021auto,JZhaoLCSS2021}, and nonholonomic agents \cite{Zhao2019tac}. In \cite{Zhao2019tac}, the reference velocity is assumed to be known by all the nonholonomic agents, and the bearing rates are needed in the control laws for double-integrator agents in \cite{Zhao2019tac} and \cite{JZhaoLCSS2021}. To drop the requirement of the bearing rates, \cite{Minh2021auto} explore{d} a bearing-only tracking control law with integral-like control terms. Bearing-only formation tracking schemes have recently {been} investigated for single-integrator \cite{Minh2022LCSS} and nonholonomic agents\cite{XLi2021Tcyb} with time-varying reference velocities. However, the directed graph of the system in \cite{Minh2022LCSS} is required to be acyclic so that the followers track the leaders in a sequential manner. In \cite{XLi2021Tcyb}, the agents estimate the reference velocity using a consensus-based observer and by exchanging the estimated variables with their neighbors. 
Nevertheless, due to the use of the signum function, the chattering effect occurs in the estimation or control systems in \cite{XLi2021Tcyb} and \cite{Minh2022LCSS}.

This letter addresses the bearing-constrained formation tracking control for nonholonomic multi-agent systems in two or three dimensions. Such a system can be used to model the kinematics of ground vehicles \cite{Zhao2019tac, Quoc2020tcns}, and underactuated autonomous underwater vehicles\cite{Xiaodong2022AutoSinica}.
As the first contribution, we propose two formation tracking control laws for the system with constant-velocity leaders and based only on the inter-agent bearings and displacements, respectively. The proposed control laws require no information exchanges between the agents, thus obviating the use of consensus-based observers. Further, the control schemes are applicable in both two and three dimensions, while in \cite{XLi2021Tcyb}, only the $2$-dimensional space is considered. Secondly, in the proposed bearing-only tracking control law, the followers do not have the information of the leaders' velocity; unlike \cite{Zhao2019tac}. By using a novel adaptive control term involving an auxiliary variable for each follower, the proposed control protocols drive the agents to the target formation and align their headings asymptotically. 

The rest of this letter is outlined as follows. Preliminaries and the problem formulation are given in Section \ref{sec:prob_formulation}. Sections \ref{sec:bearing_tracking_control_law} and \ref{sec:disp_tracking_control_law} propose two formation tracking control laws using either the inter-agent bearings or displacements. Simulation results are given in Section \ref{sec:simulation}. Section \ref{sec:conclusion} concludes this letter.

\section{Preliminaries and Problem Formulation}\label{sec:prob_formulation}
\subsubsection*{Notation}
We use $\mb{R}$ and $\mb{R}^d$ to denote the sets of real numbers and real $d$-dimensional vectors, respectively. The cross and the Kronecker products are denoted by $\times$ and $\otimes$, respectively.
We also use $\times$ to represent matrix-vector products.
Let $\bm{1}_n=[1,\ldots,1]^\top\in \mb{R}^n$, and $\bm{I}_n$ be the $n\times n$ identity matrix. For $n$ vectors or matrices $\bm{X}_1,\ldots,\bm{X}_n$, denote $\col(\bm{X}_1,\ldots,\bm{X}_n)=[\bm{X}_1^\top,\ldots,\bm{X}_n^\top]^\top$. For $n$ square matrices $\bm{A}_1,\ldots,\bm{A}_n$, let $\blkdiag(\bm{A}_1,\ldots,\bm{A}_n)$ be the block diagonal matrix with the $i$-th block diagonal being $\bm{A}_i$.

\subsection{Graph theory}
A graph is denoted by $\mc{G}=(\mc{V},\mc{E})$, where $\mc{V}=\{1,\ldots,n\}$ denotes its vertex set and $\mc{E}\subseteq\mc{V}\times \mc{V}$ the set of edges. An edge is defined by the pair $e_k=(i,j),i,j\in \mc{V}, k=1,\ldots,m,$ with $m=\vert \mathcal{E} \vert$ being the number of edges. The graph $\mc{G}$ is said to be undirected if $(i,j)\in \mc{E}$ implies $(j,i)\in \mc{E}$. The set of neighbors of $i$ is denoted by $\mc{N}_i=\{j\in\mc{V}:(i,j)\in \mc{E}\}$. For an arbitrary orientation of the $m$ edges $\{e_1,\ldots,e_m\}$ in $\mc{E}$, we define the incidence matrix $\bm{H}=[h_{ki}]\in \mb{R}^{m\times n}$ as $h_{ki}=1$ if $e_k=(j,i)$, $h_{ki}=-1$ if $e_k=(i,j)$, and $h_{ki}=0$ otherwise. 
For a connected graph $\mc{G}$, we have $\mathrm{rank}(\bm{H}(\mc{G}))=n-1$ and $\bm{H}(\mc{G})\bm{1}_n=\bm{0}$ \cite{Mesbahi2010}.
\subsection{Problem formulation} 
Consider a system of $n$ nonholonomic agents in the $d$-dimensional space ($d=2$ or $3$) with an undirected sensing graph $\mc{G}=(\mc{V},\mc{E})$. The position and velocity vectors of each agent $i$ are denoted as $\bm{p}_i$ and $\bm{v}_i\in \mb{R}^d$, respectively. The kinematic model of agent $i$ is given as \cite{Quoc2020tcns}
\begin{align*}
\dot{\bm{p}}_i&= \bm{v}_i=\m{h}_iu_i,\numberthis \label{eq:unicycle_model}\\
\dot{\bm{h}}_i&=\bm\omega_i\times\bm{h}_{i}=-\bm{h}_{i}\times \bm\omega_i,
\end{align*}
where $\m{h}_i\in\mb{R}^d$ denotes the unit heading vector, $u_i\in\mb{R}$ is the forward velocity along $\m{h}_i$, and $\bm\omega_i\in\mb{R}^d$ is the angular velocity of agent $i$. For example, in the two-dimensional ($2$-D) plane, $\bm{h}_i=[\cos(\theta_i),\sin(\theta_i)]^\top$ and $\dot{\bm{h}}_i=\omega_i[-\sin(\theta_i),\cos(\theta_i)]^\top$, where $\theta_i$ is the heading angle of agent $i$ and $\omega_i=\dot{\theta}_i\in \mb{R}$ \cite{Zhao2019tac,XLi2021Tcyb}. 
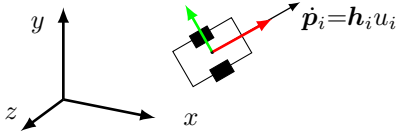
\begin{figure}[t]
\centering
\begin{tikzpicture}[scale = .9]
\node (py) at (-1,2,0) [label=below left:$y$]{};
\node (pz) at (-1,.7,2) [label=above:$z$]{};
\node (px) at (-1+1.5,0.2,0) [label=right:$x$]{};
\node[scale=0.1] (pi_dot) at (2.1,1.7) [label=right:$~~\dot{\bm{p}}_i{=}\bm{h}_iu_i$]{};
\node[place,scale=0.1] (pi) at (1.2,1.2){};

\draw[black,rotate around={30:(pi)}] (0.7,0.9) rectangle (1.7,1.5);
\filldraw[black,rotate around={30:(0.9+0.1,1.5-.1)}] (.8+0.1,1.4-.1) rectangle (1.1+0.1,1.6-.1);
\filldraw[black,rotate around={30:(0.9+0.4,1.5-.6)}] (.8+0.4,1.4-.6) rectangle (1.1+0.4,1.6-.6);
\draw[{line width=1pt}] (-1,.5,0) [frame] -- (px);
\draw[{line width=1pt}] (-1,.5,0) [frame] -- (py);
\draw[{line width=1pt}] (-1,.5,0) [frame] -- (pz);
\draw[{line width=.5pt}] (pi) [frame] -- (2.5,1.94);
\draw[{line width=1.pt},green,->] (pi)[frame]  -- (.8,1.9);
\draw[{line width=1.pt},red,->] (pi) [frame] -- (pi_dot);
\end{tikzpicture}
\caption{A nonholonomic agent in $\mb{R}^d,d=2$ or $3$.}
\label{fig:unicycle}
\end{figure}

Assume that there are $n_l\geq 2$ leaders; the other $n_f:=n-n_l$ agents are called \textit{followers}. The index set\textcolor{blue}{s} of the leaders and followers are $\mc{V}_l=\{1,\ldots,n_l\}$ and $\mc{V}_f=\mc{V}\setminus \mc{V}_l$, respectively. Denote $\bm{p}_f=\col(\bm{p}_{n_l+1},\ldots,\bm{p}_n)\in \mb{R}^{dn_f}$ and $\bm{p}=\mathrm{col}(\bm{p}_1,\ldots,\bm{p}_n)\in \mb{R}^{dn}$. The \textit{augmented graph} $\bar{\mc{G}}=(\mc{V},\bar{\mc{E}})$ is obtained from $\mc{G}$ by insertion of an edge between every two leader nodes, i.e., $\bar{\mc{E}}=\mc{E}\cup\{(i,j):i,j\in \mc{V}_l,i\neq j\}$.
A formation, $(\bar{\mc{G}},\bm{p})$, consists of the graph $\bar{\mc{G}}$ and the configuration $\bm{p}$.  
For each $(i,j)\in \bar{\mc{E}}$, the corresponding \textit{displacement} vector is  $\bm{z}_{ij}=\bm{p}_j-\bm{p}_i\in\mb{R}^d$. Consider an orientation of the edges in $\bar{\mc{E}}$ and let $\bm{H}\in \mb{R}^{m\times n}$ be the corresponding incidence matrix. Then, $\bm{z}=\col(\bm{z}_1,\ldots,\bm{z}_m)=(\bm{H}\otimes \bm{I}_d)\bm{p}:=\bar{\bm{H}}\bm{p}$.

Suppose that $\bm{p}_i\neq \bm{p}_j$. The bearing vector $\bm{g}_{ij}\in\mb{R}^d$ from agent $i$ to a neighboring agent $j$ is defined as
\begin{equation}
\bm{g}_{ij}=\frac{\bm{z}_{ij}}{\abs{\bm{z}_{ij}}}=\frac{\bm{p}_j-\bm{p}_i}{\abs{\bm{p}_j-\bm{p}_i}}.
\end{equation}
The orthogonal projection matrix associated with $\bm{g}_{ij}$ is defined as $\bm{P}_{\bm{g}_{ij}}=\bm{I}_d-\bm{g}_{ij}\bm{g}_{ij}^\top\in \mb{R}^{d\times d}$. 
We define $\bm{g}=\mathrm{col}(\bm{g}_1,\ldots,\bm{g}_m)\in \mb{R}^{dm}$ corresponding to the order of the edges in $\mc{E}$, i.e., $\bm{g}_{k}=\bm{g}_{ij}$ whenever the edge $\bm{z}_k=\bm{z}_{ij}$.

The moving \textit{target formation} of the system is $\bm{p}^*(t)=\col(\bm{p}_1^*(t),\ldots,\bm{p}_n^*(t))\in \mb{R}^{dn}$, whose formation pattern is specified by a time-invariant set of desired bearings 
\begin{equation}\big\{\bm{g}_{ij}^*:\bm{g}_{ij}^*=(\bm{p}_j^*-\bm{p}_i^*)/\abs{\bm{p}_j^*-\bm{p}_i^*},(i,j)\in \bar{\mc{E}}\big\}.
\end{equation} 
Each leader $i\in \mc{V}_l$ can follow the predefined trajectory $\bm{p}_i^*(t)=\bm{p}_i^*(0)+\bm{v}_ct$ with the same constant velocity $\bm{v}_c:=u_c\bm{h}_c\in \mb{R}^d$, where $|u_c|=\abs{\bm{v}_c}$ is a constant and $\bm{h}_c\in \mb{R}^d$ is a time-invariant heading. Thus, $\dot{\bm{p}}_i=\bm{v}_c$, $\bm{h}_i=\bm{h}_c$ and $\bm\omega_i=\bm{0}$ for all $i\in \mc{V}_l$, {and the desired bearings between the leaders are maintained for all} $t\geq 0$. However, each follower $i\in \mc{V}_f$ is unaware of the desired trajectory $\bm{p}_i^*(t)$ nor the reference velocity $\bm{v}_c$, but only knows the desired bearings to its neighbors $\{\bm{g}_{ij}^*\}_{j\in \mc{N}_i}$. Let $\bar{\bm p}^*(t)=\sum_{i=1}^n\bm{p}_i^*(t)/n$ be the target \textit{formation's centroid}, and $\tilde{\bm p}^*=\bm{p}^*(t)-\bm{1}_n\otimes \bar{\bm p}^*(t)$ the geometric pattern of the formation expressed with regard to $\bar{\bm p}^*(t)$. Thus, given a target formation $\bm{p}^*(t)$, $\tilde{\bm p}^*$ is a constant vector.

 We adopt the following assumption for the uniqueness of the target formation. 
\begin{Assumption}\label{ass:bearing_rigid}
The desired formation $(\bar{\mc{G}},\bm{p}^*)$ is infinitesimally bearing rigid.
\end{Assumption}

Though the graph $\bar{\mc{G}}$ is undirected, the leaders are controlled independently and do not necessarily sense relative information to their neighbors. Given $(\bar{\mc{G}},\bm{p}^*)$, the bearing Laplacian $\bm{\mc{B}}(\bm{p}^*)\in \mb{R}^{dn\times dn}$ is define as $[\bm{\mc{B}}]_{ij}=-\bm{P}_{\bm{g}_{ij}^*},i\neq j,(i,j)\in \bar{\mc{E}})$, $[\bm{\mc{B}}]_{ii}=\sum_{k\in \mc{N}_i}\bm{P}_{\bm{g}_{ik}^*}$, and $[\bm{\mc{B}}]_{ij}=\bm 0$ otherwise. We can partition $\bm{\mc{B}}$ into
\begin{equation}
\bm{\mc{B}}=\begin{bmatrix}
\bm{\mc{B}}_{ll} &\bm{\mc{B}}_{lf}\\
\bm{\mc{B}}_{fl} &\bm{\mc{B}}_{ff}
\end{bmatrix},
\end{equation} where $\bm{\mc{B}}_{ll}\in \mb{R}^{dn_l\times dn_l}$ and $\bm{\mc{B}}_{ff}\in \mb{R}^{dn_f\times dn_f}$. Under Assumption \ref{ass:bearing_rigid} and $n_l\geq 2$, $\mathrm{null}(\bm{\mc{B}})=\mathrm{span}(\bm{1}_n\otimes \bm{I}_d,\bm p^*)$ and $\bm{\mc{B}}_{ff}$ is a positive definite matrix \cite{Zhao2016tac}.

Define the tracking error vectors $\bm\delta_p(t)=\bm{p}(t)-\bm{p}^*(t)\in \mb{R}^{dn}$ and $\bm\delta_{p_f}(t)=\bm{p}_f(t)-\bm{p}_f^*(t)\in \mb{R}^{dn_f}$. Thus, $\bm\delta_p=\col(\bm{0}_{dn_l},\bm\delta_{p_f})$ since $\bm p_i(t)=\bm p_i^*(t),\forall i \in \mc{V}_l$. Suppose that no collision happens (sufficient conditions for collision avoidance will be given later) and no communication occurred between the agents. We can now state the formation tracking control problem under study.
\begin{Problem}
Under Assumption \ref{ass:bearing_rigid}, design a distributed control law $(u_i,\bm{\omega}_i)$ for each follower $i\in \mc{V}_f$ so that $\bm\delta_p(t)\rightarrow \bm 0$ asymptotically with two or more constant velocity leaders and using either merely the inter-agent bearing vectors $\{\bm{g}_{ij}(t)\}_{j\in \mc{N}_i}$ or the displacements $\{\bm{z}_{ij}\}_{j\in \mc{N}_i}$.
\end{Problem}

\begin{Remark}
The reference velocity $\bm{v}_c$ may be communicated to the followers, e.g., along a spanning tree of the graph $\mc{G}$. Alternatively, it can be estimated using consensus-based distributed observers by exchanging auxiliary variables between neighboring agents \cite{YHuang2021tcns,XLi2021Tcyb}. However, these approaches are not applicable in the problem under study, since we assume no communication occurred between the agents.
\end{Remark}

\section{Bearing-only Formation Tracking Control}\label{sec:bearing_tracking_control_law}
This section studies bearing-only formation tracking control for nonholonomic agents. Each follower $i\in \mc{V}_f$ senses only the bearing vectors $\bm{g}_{ij}(t)$, and is provided the desired bearing vectors $\bm{g}_{ij}^*$ with regard to its neighbors $j\in \mc{N}_i$.
\subsection{Proposed formation tracking control protocol}
Define $\bm{r}_i:=\sum_{j\in \mc{N}_i}(\bm{g}_{ij}-\bm{g}_{ij}^*)$ for each $i\in \mc{V}_f$. We propose the following control law for each $i\in \mc{V}_f$
\begin{equation}\label{eq:tracking_control_law_i}
\begin{cases}
u_i=\bm{h}_i^\top( k_1\bm{r}_i+\bm{\xi}_i )\\
\dot{\bm{\xi}}_i=\bm{h}_i\bm{h}_i^\top \bm{r}_i-(\bm{I}_d-\bm{h}_i\bm{h}_i^\top )\bm{\xi}_i\\
\bm\omega_i=\bm{h}_i\times k_2(\bm{r}_i+\bm{\xi}_i),
\end{cases}
\end{equation}
where $k_1$ and $k_2>0$, and $\bm{\xi}_i\in \mb{R}^d$ is an auxiliary vector associated with each agent $i$ with $\bm{\xi}_i(0)=\bm{0}$ {(or otherwise chosen arbitrarily)}. Note that $\bm{h}_i\bm{h}_i^\top$ and $(\bm{I}_d-\bm{h}_i\bm{h}_i^\top )$ are the orthogonal projections onto $\bm{h}_i$ and its orthogonal complement, respectively. We have two observations on the design of the adaptive control in \eqref{eq:tracking_control_law_i} involving $\bm{\xi}_i$, which will be shown to converge to $\bm{v}_c$. First, considering the nonholonomic constraint \eqref{eq:unicycle_model}, the first term in the expression of $\dot{\bm{\xi}}_i(t)$ associated with the bearing error vector $\bm{r}_i$ acts only in the feasible velocity direction $\bm{h}_i$. Second, the second term in the expression of $\dot{\bm{\xi}}_i(t)$ is a negative $\bm{\xi}_i$-variable feedback with the matrix gain being the orthogonal projection onto $\{\mathrm{span}(\bm{h}_i)\}^\perp$. This helps to steer $\bm\xi_i$ to the direction parallel to $\bm{h}_i$, $\forall i\in \mc{V}_f$, as will be shown in the main analysis below.
\subsection{Stability analysis}
Substituting the preceding control law into \eqref{eq:unicycle_model} gives 
\begin{equation}\label{eq:traject_under_control_law_i}
\begin{cases}
\dot{\bm p}_i=\bm{h}_i\bm{h}_i^\top( k_1\bm{r}_i+\bm{\xi}_i )\\
\dot{\bm{\xi}}_i=\bm{h}_i\bm{h}_i^\top \bm{r}_i-(\bm{I}_d-\bm{h}_i\bm{h}_i^\top )\bm{\xi}_i\\
\dot{\bm{h}}_i=-\bm{h}_i\times\bm{h}_i\times k_2(\bm{r}_i+\bm{\xi}_i)\\\quad=(\bm{I}_d-\bm{h}_i\bm{h}_i^\top)k_2(\bm{r}_i+\bm{\xi}_i),
\end{cases}
\end{equation} where we have used the relation $-\bm{x}\times\bm{x}\times \bm{y}=(\bm{I}_d-\bm{x}\bm{x}^\top)\bm{y}$ for any $\bm{x},\bm{y}\in \mb{R}^d$ \cite{Ma2004}.
We define the following stacked vectors and matrices
\begin{align*}
&\bm\xi_f=\col(\bm{\xi}_{n_l+1},\ldots,\bm{\xi}_{n}), \bm\xi=\col(\bm{1}_{n_l}\otimes \bm{v}_c,\bm{\xi}_f),\\
 &\bm{Z}=\blkdiag(\bm{0}_{dn_l\times dn_l},\bm{I}_{dn_f})\in \mb{R}^{dn\times dn},\\
&\bm{h}=\col(\bm{h}_1,\ldots,\bm{h}_n)\in \mb{R}^{dn},\numberthis \label{eq:stack_variable_def}\\ 
 &\bm{D}_{\bm{h}_i}=\blkdiag(\{\bm{h}_i\bm{h}_i^\top\}_{i\in \mc{V}}) \in \mb{R}^{dn\times dn},\\
&\bm{D}_{\bm{h}_i^\perp}=\blkdiag(\{\bm{I}_d-\bm{h}_i\bm{h}_i^\top\}_{i\in \mc{V}})\in \mb{R}^{dn\times dn}.
\end{align*}
In a compact form, we can express the $n$-agent system as
\begin{align*}
\dot{\bm{p}}&=\col(
\bm{1}_{n_l}\otimes \bm{v}_c,
\bm{0}_{dn_f}
)-
\blkdiag(\bm{0}_{dn_l\times dn_l},\{\bm{h}_i\bm{h}_i^\top\}_{i\in \mc{V}_f})\\
&\quad\times\left(k_1\bar{\bm{H}}^\top(\bm{g}-\bm{g}^*)-\bm\xi\right)\\
&=(\bm{I}_{dn}-\bm{Z})(\bm{1}_n\otimes \bm{v}_c)-\bm{Z}\bm{D}_{\bm{h}_i}\left(k_1\bar{\bm{H}}^\top(\bm{g}-\bm{g}^*)-\bm\xi\right)\\
\dot{\bm\xi}&=-\blkdiag\left(\bm{0}_{dn_l\times dn_l},\{\bm{h}_i\bm{h}_i^\top\}_{i\in \mc{V}_f}\right)\bar{\bm{H}}^\top(\bm{g}-\bm{g}^*)\\
&\quad-\blkdiag\left(\bm{0}_{dn_l\times dn_l},\{\bm{I}_d-\bm{h}_i\bm{h}_i^\top\}_{i\in \mc{V}_f}\right)\bm\xi\\
&=-\bm{Z}\bm{D}_{\bm{h}_i}\bar{\bm{H}}^\top(\bm{g}-\bm{g}^*)-\bm{Z}\bm{D}_{\bm{h}_i^\perp}\bm\xi, \numberthis \label{eq:compact_form}\\
\dot{\bm h}&=-\blkdiag\left(\bm{0}_{dn_l\times dn_l},\{\bm{I}_d-\bm{h}_i\bm{h}_i^\top\}_{i\in \mc{V}_f}\right)\\
&\qquad\times k_2\left(\bar{\bm{H}}^\top(\bm{g}-\bm{g}^*)-\bm\xi\right)\\
&=-\bm{Z}\bm{D}_{\bm{h}_i^\perp}k_2\left(\bar{\bm{H}}^\top(\bm{g}-\bm{g}^*)-\bm\xi\right).
\end{align*}

We first have the following lemma.
\begin{Lemma}\label{lm:bearing_pos_semidef_ineqs}\cite[Lem. 2]{Zhao2019tac}
There holds that
\begin{align}
\bm{p}^\top\bar{\bm H}^\top(\bm{g}-\bm{g}^*)&\geq 0,\label{eq:bearing_pos_semidef_ineq1}\\
(\bm{p}-\bm{p}^*)^\top\bar{\bm H}^\top(\bm{g}-\bm{g}^*)&\geq 0,\label{eq:bearing_pos_semidef_ineq2}
\end{align} where, when no agents coincide in $\bm{p}^*$ or $\bm{p}$, the equalities hold if and only if $\bm{g}=\bm{g}^*$.
\end{Lemma}

Based on Lyapunov stability theory, Barbalat's lemma and the invariance principle, the convergence of the system \eqref{eq:compact_form} is then shown in the following theorem.
\begin{Theorem}\label{thm:bearing_only_asympt_converg}
Suppose that Assumption \ref{ass:bearing_rigid} holds. Under the tracking control law \eqref{eq:tracking_control_law_i}, $\bm p(t)\rightarrow \bm p^*(t)$ {globally and} asymptotically as $t\rightarrow \infty$.
\end{Theorem}
\begin{proof}
Consider the Lyapunov function \begin{align*} \label{eq:Lyapunov_funct}
V=\bm{z}^\top&(\bm{g}-\bm{g}^*)+\frac{1}{2}\abs{\bm\xi-\bm{1}_n\otimes\bm{v}_c}^2\\
&+\frac{u_c}{(2k_2)}\abs{\bm{h}-\bm{1}_n\otimes \bm{h}_c}^2,\numberthis
\end{align*} where $u_c=\abs{\bm{v}_c}$.
Note that $\bm{z}^\top(\bm{g}-\bm{g}^*)=\bm{p}^\top\bar{\bm H}^\top(\bm{g}-\bm{g}^*)\geq 0$ and $\bm{z}^\top(\bm{g}-\bm{g}^*)=0$ if and only if $\bm{g}=\bm{g}^*$ (Lemma \ref{lm:bearing_pos_semidef_ineqs}), or equivalently $\bm{p}=\bm{p}^*$ under Assumption \ref{ass:bearing_rigid}. $V(t)=0$ if and only if $\bm{p}_f=\bm{p}_f^*$, $\bm{\xi}=\bm{1}_{n}\otimes\bm{v}_c$, and $\bm{h}=\bm{1}_{n}\otimes\bm{h}_c$.

The derivative of $V$ along the trajectory of \eqref{eq:compact_form} is given as
\begin{align*}
\dot{V}=&~\bm{z}^\top\dot{\bm{g}}+(\bm{g}-\bm{g}^*)^\top\bar{\bm H}\dot{\bm p}+(\bm\xi-\bm{1}_n\otimes\bm{v}_c)^\top\dot{\bm\xi}\\
&\qquad+\frac{u_c}{k_2}(\bm{h}-\bm{1}_n\otimes \bm{h}_c)^\top\dot{\bm h}\\
=&~(\bm{g}-\bm{g}^*)^\top\bar{\bm H}(\bm{I}_{dn}-\bm{Z})(\bm{1}_{n}\otimes \bm{v}_c)\\
&-(\bm{g}-\bm{g}^*)^\top\bar{\bm H}\bm{Z}\bm{D}_{\bm{h}_i}k_1\bar{\bm{H}}^\top(\bm{g}-\bm{g}^*)
\\
&+(\bm{g}-\bm{g}^*)^\top\bar{\bm H}\bm{Z}\bm{D}_{\bm{h}_i}\bm{\xi}-\bm\xi^\top\bm{Z}\bm{D}_{\bm{h}_i}\bar{\bm{H}}^\top(\bm{g}-\bm{g}^*)\\
&-\bm\xi^\top\bm{ZD}_{\bm{h}_i^\perp}\bm\xi+(\bm{1}_n\otimes\bm{v}_c)^\top\bm{Z}\bm{D}_{\bm{h}_i}\bar{\bm{H}}^\top(\bm{g}-\bm{g}^*)\\
&+(\bm{1}_{n}\otimes\bm{v}_c)^\top\bm{ZD}_{\bm{h}_i^\perp}\bm\xi
+(\bm{1}_n\otimes\bm{v}_c)^\top\bm{ZD}_{\bm{h}_i^\perp}\bar{\bm{H}}^\top\\
&\times(\bm{g}-\bm{g}^*)-(\bm{1}_{n}\otimes\bm{v}_c)^\top\bm{ZD}_{\bm{h}_i^\perp}\bm\xi\\
=&~(\bm{1}_{n}\otimes \bm{v}_c)^\top(\bm{I}_{dn}-\bm{Z})\bar{\bm{H}}^\top(\bm{g}-\bm{g}^*)\\
&-(\bm{g}-\bm{g}^*)^\top\bar{\bm H}\bm{Z}\bm{D}_{\bm{h}_i}k_1\bar{\bm{H}}^\top(\bm{g}-\bm{g}^*)-\bm\xi^\top\bm{ZD}_{\bm{h}_i^\perp}\bm\xi\\
&+(\bm{1}_n\otimes\bm{v}_c)^\top\bm{Z}(\bm{D}_{\bm{h}_i^\perp}+\bm{D}_{\bm{h}_i})\bar{\bm{H}}^\top(\bm{g}-\bm{g}^*), \numberthis
\end{align*}
where, in the second equality we have used 
\begin{align*}
\bm{z}^\top\dot{\bm{g}}&=\bm{z}^\top\blkdiag\big({\bm{P}_{\bm{g}_1}}/{\abs{\bm{z}_1}},\ldots,{\bm{P}_{\bm{g}_m}}/{\abs{\bm{z}_m}}\big)\dot{\bm z}\\
&=\blkdiag\big({\bm{z}_1^\top\bm{P}_{\bm{g}_1}}/{\abs{\bm{z}_1}},\ldots,{\bm{z}_m^\top\bm{P}_{\bm{g}_m}}/{\abs{\bm{z}_m}}\big)\dot{\bm z}\\
&=0
\end{align*} and $\bm{h}^\top\dot{\bm h}=(1/2)\frac{d}{dt}\abs{\bm h}^2=0$, and the last equality follows from $\bar{\bm H}\bm{Z}\bm{D}_{\bm{h}_i}=(\bm{Z}\bm{D}_{\bm{h}_i}\bar{\bm{H}}^\top)^\top$. Using the facts that $\bm{D}_{\bm{h}_i}+\bm{D}_{\bm{h}_i^\perp}=\bm{I}_{dn}$ and $\bar{\bm H}(\bm{1}_n\otimes\bm{v}_c)=\bm{0}$ (since $\mathrm{null}(\bar{\bm H})=\mathrm{span}(\bm{1}_n\otimes \bm{I}_d)$), we further obtain that
\begin{align*}
\dot{V}&=-k_1(\bm{g}-\bm{g}^*)^\top\bar{\bm H}\bm{Z}\bm{D}_{\bm{h}_i}\bar{\bm{H}}^\top(\bm{g}-\bm{g}^*)-\bm\xi^\top\bm{ZD}_{\bm{h}_i^\perp}\bm\xi\\
&=-\sum_{i\in\mc{V}_f}\left(k_1\bm{r}_i^\top\bm{h}_i\bm{h}_i^\top\bm{r}_i+\bm{\xi}_i^\top(\bm{I}_d-\bm{h}_i\bm{h}_i^\top )\bm{\xi}_i\right)\\
&\leq 0. \numberthis \label{eq:dot_V}
\end{align*}
It follows that $V(t)\leq V(0)$ for all $t$. Thus, $\bm{z}^\top(\bm{g}-\bm{g}^*)$ and $\bm\xi(t)$ are always bounded. By this and the quadratic inequality 
\begin{equation}\label{eq:quadratic_ineq_in_pos_err}
2\abs{\bar{\bm H}}(\abs{\bm{\delta}_p}+\abs{\tilde{\bm p}^*})\bm{z}^\top(\bm{g}-\bm{g}^*)\geq {\lambda_{\min}(\bm{\mc{B}}_{ff})\abs{\bm{\delta}_p}^2}{}
\end{equation} in the position error $\abs{\bm{\delta}_p}$ \cite[Coroll. 2]{Zhao2019tac}, we have that $\abs{\bm\delta_{p_f}(t)}$ is always bounded. As a result, by LaSalle's invariance principle \cite{Khalil2002}, $\bm\delta_{p_f}(t)$ and $\bm\xi_f(t)$ asymptotically converge to the invariant set where $\dot{V}=0$. From $\dot{V}=0$, and the idempotence property of projection matrices that $\bm{Z}\bm{D}_{\bm{h}_i}=(\bm{Z}\bm{D}_{\bm{h}_i})^2$ and $\bm{ZD}_{\bm{h}_i^\perp}=(\bm{ZD}_{\bm{h}_i^\perp})^2$, one has
\begin{align}
&\bm{Z}\bm{D}_{\bm{h}_i}\bar{\bm{H}}^\top(\bm{g}-\bm{g}^*)=\bm{0}, \label{eq:dot_V_invariant1}\\
& \bm{ZD}_{\bm{h}_i^\perp}\bm\xi=\bm{0}.\label{eq:dot_V_invariant2}
\end{align}
Substituting the preceding equalities into \eqref{eq:compact_form} gives
\begin{align*}
\dot{\bm p}&\stackrel{\eqref{eq:dot_V_invariant2}}{=}\col\left(\bm{1}_{n_l}\otimes \bm{v}_c,\bm{0}_{dn_f}\right)+\bm{Z}\bm{D}_{\bm{h}_i}\bm\xi+\bm{ZD}_{\bm{h}_i^\perp}\bm\xi\\
&=\col\left(\bm{1}_{n_l}\otimes \bm{v}_c,\bm\xi_f\right), \numberthis\label{eq:dot_p_inf}\\
\dot{\bm \xi}&=\bm{0},\numberthis \label{eq:dot_xi_inf}\\
\dot{\bm h}&=-k_2\bm{Z}\bm{D}_{\bm{h}_i^\perp}\bar{\bm{H}}^\top(\bm{g}-\bm{g}^*)\stackrel{\eqref{eq:dot_V_invariant1}}{=} -k_2\bm{Z}\bar{\bm{H}}^\top(\bm{g}-\bm{g}^*).\numberthis\label{eq:dot_h_inf}
\end{align*}
It follows from \eqref{eq:dot_p_inf} and \eqref{eq:dot_xi_inf} that the followers' velocities {satisfy} $\bm{v}_i=\bm{\xi}_i, \forall i\in \mc{V}_f$, {and} are all constant vectors. Thus, for any $k\in\mc{V}_f$ with $\bm{\xi}_k\neq \bm 0$, its heading $\bm{h}_k={\bm{v}_k}/{\abs{\bm{v}_k}}$ is time-invariant, and hence $\dot{\bm{h}}_k(t)\rightarrow\bm{0}$ by Barbalat's lemma (from \eqref{eq:traject_under_control_law_i}, $\ddot{\bm{h}}_k(t)$ is bounded due to the boundedness of $\dot{\bm h}_k,\dot{\bm r}_k$ and $\dot{\bm\xi}_k$). Next, suppose that $\bm{\xi}_i=\bm 0$ for certain $i\in \mc{V}_f$, i.e., agent $i$ does not translate. If $\bm r_i \equiv\bm 0$ for such an agent then $\dot{\bm{h}}_i\stackrel{\eqref{eq:dot_h_inf}}{=}k_2\bm r_i\equiv\bm{0}$. Otherwise, let us assume $\bm r_i \neq \bm 0$. Then, by \eqref{eq:dot_V_invariant1} and \eqref{eq:tracking_control_law_i}, one has $\bm h_i^\top \bm r_i=0$ and $\bm\omega_i=k_2(\bm{h}_i\times\bm{r}_i)$, respectively. That is, $\bm{h}_i$ keeps rotating while always maintaining that $\bm h_i \perp \bm r_i$. This is not possible since geometrically, the control $\bm\omega_i=k_2(\bm{h}_i\times\bm{r}_i)$ keeps steering $\bm{h}_i$ to align with $\bm{r}_i$. Thus, in either case, one has $\dot{\bm{h}}_i=\bm 0$. By $\dot{\bm h}=\bm 0$ and \eqref{eq:dot_h_inf}, we obtain that $\bm{Z}\bar{\bm{H}}^\top(\bm{g}-\bm{g}^*)=\bm 0$. This further leads to 
\begin{equation}
\bm\delta_p\bm{Z}\bar{\bm{H}}^\top(\bm{g}-\bm{g}^*)=\bm 0 \Leftrightarrow \bm\delta_p\bar{\bm{H}}^\top(\bm{g}-\bm{g}^*)=\bm 0.
\end{equation} 
It then follows from Lemma \ref{lm:bearing_pos_semidef_ineqs} that $\bm{g}=\bm{g}^*$ {(since no agents coincide in $\bm{p}^*$)}, or equivalently, the agents achieve the target formation {globally and} asymptotically. Note that in order to satisfy $\bm{g}=\bm{g}^*$, or i.e., the desired formation pattern is maintained, the constant velocities of the agents must be the same, i.e., $\bm{v}_i=u_i\bm{h}_i=\bm{v}_c,\forall i\in \mc{V}_f$. This also implies that $\bm{h}_i\rightarrow\bm{h}_c$ {and $\bm\xi_i\rightarrow \bm{v}_c$ asymptotically} as $t\rightarrow \infty,\forall i\in\mc{V}_f$.
\end{proof}

\subsection{Sufficient condition for collision avoidance}\label{subsec:collision_avoid}
Let $\kappa$ be the desired minimum distance between any two agents satisfying $0<\kappa<\min_{i,j\in\mc{V}}\abs{\bm{p}_i^*-\bm{p}_j^*}$, and 
\begin{equation}\label{eq:epsilon}
\epsilon:=({1}/{\sqrt{n}})(\min_{i,j\in\mc{V}}\abs{\bm{p}_i^*-\bm{p}_j^*}-\kappa).
\end{equation}
 Following a similar argument as in \cite[Coroll. 3]{Zhao2019tac}, we now give a sufficient condition for inter-agent collision avoidance. 
\begin{Corollary}\label{coroll:no_collision} Suppose that Assumption \ref{ass:bearing_rigid} holds. Under the tracking control law \eqref{eq:tracking_control_law_i},
if initially $V(0)\leq \beta$ for a sufficiently small constant $\beta>0$ such that
\begin{equation}
(\gamma\beta+\sqrt{\gamma^2\beta^2+4\gamma\beta\abs{\tilde{\bm p}^*}})/2\leq\epsilon,
\end{equation} where $\gamma=(2\abs{\bar{\bm H}})/\lambda_{\min}(\bm{\mc{B}}_{ff})$, then $\forall i,j\in \mc{V}, i\neq j$, there holds $\abs{\bm{p}_i-\bm{p}_j}\geq \kappa$ for all $t\geq 0$.
\end{Corollary}
\begin{proof}
It follows from $\dot{V}\leq 0$ that $V(t)\leq \beta$ for all $t\geq 0$. Thus, $\bm{z}^\top(\bm{g}-\bm{g}^*)\leq \beta,\forall t\geq 0$. By \eqref{eq:quadratic_ineq_in_pos_err}, one therefore has that
\begin{equation*}
{\abs{\bm{\delta}_p}^2}\leq \beta\gamma(\abs{\bm{\delta}_p}+\abs{\tilde{\bm p}^*}),
\end{equation*}
where $\gamma = (2\abs{\bar{\bm H}})/\lambda_{\min}(\bm{\mc{B}}_{ff})$. The preceding inequality indicates that $\abs{\bm{\delta}_p}\in [0,\phi]$, where $\phi =(\gamma\beta+\sqrt{\gamma^2\beta^2+4\gamma\beta\abs{\tilde{\bm p}^*}})/2$. It then follows that $\phi\leq \epsilon$ for a sufficiently small constant $\beta>0$. Consequently, $\abs{\bm{\delta}_p}\leq\epsilon$ for all $t\geq 0$ since it always holds that $V(t)\leq \beta$ if initially $V(0)\leq \beta$. We thus obtain
\begin{align*}
\abs{\bm{p}_i-\bm{p}_j}&=\abs{\bm{p}_i-\bm{p}_i^*-(\bm{p}_j-\bm{p}_j^*)+\bm{p}_i^*-\bm{p}_j^*}\\
&\geq \abs{\bm{p}_i^*-\bm{p}_j^*}-\abs{\bm{p}_i-\bm{p}_i^*}-\abs{\bm{p}_j-\bm{p}_j^*}\\
&\geq  \abs{\bm{p}_i^*-\bm{p}_j^*}-\textstyle\sum_{k=1}^n\abs{\bm{p}_k-\bm{p}_k^*}\\
&\geq \abs{\bm{p}_i^*-\bm{p}_j^*}-\sqrt{n}\abs{\bm\delta_{p}}\\
&\stackrel{\eqref{eq:epsilon}}{\geq} \kappa,\numberthis
\end{align*}
for all $t\geq 0$ and for all $i,j\in \mc{V}$. This shows that no agents collide with each other. As a result, by Theorem \ref{thm:bearing_only_asympt_converg} the agents achieve the target formation asymptotically.
\end{proof}
\section{Bearing-based Formation Tracking Control}\label{sec:disp_tracking_control_law}
The bearing-based formation tracking control is investigated in this section. Each follower $i\in \mc{V}_f$ measures only the displacement vectors $\bm{z}_{ij}=\bm p_j - \bm p_i$, and knows the desired bearings $\bm{g}_{ij}^*$ with regard to its neighbors $j\in \mc{N}_i$. 
\subsection{Proposed formation tracking control protocol}
Let $\bm{r}_i=-\sum_{j\in \mc{N}_i}\bm{P}_{\bm{g}_{ij}^*}(\bm{p}_i-\bm{p}_j)$ for each $i\in \mc{V}_f$. We propose the following formation tracking control law
\begin{equation}\label{eq:disp_tracking_control_law_i}
\begin{cases}
u_i=\bm{h}_i^\top( k_1\bm{r}_i+\bm{\xi}_i )\\
\dot{\bm{\xi}}_i=\bm{h}_i\bm{h}_i^\top \bm{r}_i-(\bm{I}_d-\bm{h}_i\bm{h}_i^\top )\bm{\xi}_i\\
\bm\omega_i=\bm{h}_i\times k_2(\bm{r}_i+\bm{\xi}_i),
\end{cases}
\end{equation}
where $k_1,k_2>0$, and $\bm{\xi}_i(0){\in \mb{R}^d}$ for all $i\in \mc{V}_f$. {The preceding control law has a similar structure as \eqref{eq:tracking_control_law_i}, but with control vectors $\bm r_i$ designed based on the displacements $\bm z_{ij}$. Thus, \eqref{eq:disp_tracking_control_law_i} utilizes in addition the inter-agent distances. In the sequel, we show that the controller \eqref{eq:disp_tracking_control_law_i} steers the nonholonomic agents to the target formation asymptotically. Therefore, it solves the bearing-based formation tracking for nonholonomic agents, compared with \cite{Zhao2017,YHuang2021tcns}.}

Substituting the preceding control law into \eqref{eq:unicycle_model} and by collecting them for all agents, we get
\begin{equation}\label{eq:disp_compact_form}
\begin{cases}
\dot{\bm p}&=(\bm{I}_{dn}-\bm{Z})(\bm{1}_n\otimes \bm{v}_c)-\bm{Z}\bm{D}_{\bm{h}_i}(k_1\mc{\bm B}\bm\delta_p-\bm\xi),\\
\dot{\bm\xi}&=-\bm{Z}\bm{D}_{\bm{h}_i}\mc{\bm B}\bm\delta_p-\bm{Z}\bm{D}_{\bm{h}_i^\perp}\bm\xi\\
\dot{\bm h}&=-\bm{Z}\bm{D}_{\bm{h}_i^\perp}k_2(\mc{\bm B}\bm\delta_p-\bm\xi)
\end{cases}
\end{equation}
where $\bm\delta_p=\bm{p}-\bm{p}^*$, and the matrices $\bm{Z},\bm{D}_{\bm{h}_i^\perp},\bm{D}_{\bm{h}_i}$ and $\bm \xi$ are defined as in \eqref{eq:stack_variable_def}.
\subsection{Stability analysis}
We have the following theorem, which is the main result of this section.
\begin{Theorem}\label{thm:bearing_based_tracking_control}
Suppose that Assumption \ref{ass:bearing_rigid} holds. Under the tracking control law \eqref{eq:disp_tracking_control_law_i}, $\bm p(t)\rightarrow \bm p^*(t)$ {globally and} asymptotically as time diverges.
\end{Theorem}
\begin{proof}
Consider the Lyapunov function \begin{align*} \label{eq:Lyapunov_funct_disp}
V=&~\frac{1}{2}\bm\delta_p^\top\mc{\bm B}\bm\delta_p+\frac{1}{2}\abs{\bm\xi-\bm{1}_n\otimes\bm{v}_c}^2\\
&+\frac{u_c}{(2k_2)}\abs{\bm{h}-\bm{1}_n\otimes \bm{h}_c}^2.\numberthis
\end{align*}
It is noted that $\bm\delta_p^\top\mc{\bm B}\bm\delta_p=\bm\delta_{p_f}^\top\mc{\bm B}_{ff}\bm\delta_{p_f}\geq 0$, where the equality holds if and only if $\bm\delta_{p_f}=\bm 0 \Leftrightarrow \bm p_f=\bm p_f^*$. Thus, $V$ is continuously differentiable, positive definite and radially unbounded. The derivative of $V$ along the trajectory of \eqref{eq:disp_compact_form} is given as
\begin{align*}
\dot{V}=&~\bm\delta_p^\top\mc{\bm B}(\dot{\bm{p}}-\bm{1}_n\otimes \bm v_c)+(\bm\xi-\bm{1}_n\otimes\bm{v}_c)^\top\dot{\bm\xi}\\
&\qquad+\frac{u_c}{k_2}(\bm{h}-\bm{1}_n\otimes \bm{h}_c)^\top\dot{\bm h}\\
=&~\bm\delta_p^\top\mc{\bm B}(\bm{I}_{dn}-\bm{Z})(\bm{1}_n\otimes \bm{v}_c)-\bm\delta_p^\top\mc{\bm B}\bm{Z}\bm{D}_{\bm{h}_i}k_1\mc{\bm B}\bm\delta_p\\
&+\bm\delta_p^\top\mc{\bm B}\bm{Z}\bm{D}_{\bm{h}_i}\bm\xi-\bm\xi^\top\bm{Z}\bm{D}_{\bm{h}_i}\mc{\bm B}\bm\delta_p-\bm\xi^\top\bm{Z}\bm{D}_{\bm{h}_i^\perp}\bm\xi\\
&+(\bm{1}_n\otimes \bm v_c)^\top\bm{Z}\bm{D}_{\bm{h}_i}\mc{\bm B}\bm\delta_p+(\bm{1}_n\otimes \bm v_c)^\top\bm{Z}\bm{D}_{\bm{h}_i^\perp}\bm\xi\\
&+(\bm{1}_n\otimes\bm{v}_c)^\top\bm{Z}\bm{D}_{\bm{h}_i^\perp}(\mc{\bm B}\bm\delta_p-\bm\xi)\\
=&-\bm\delta_p^\top\mc{\bm B}\bm{Z}\bm{D}_{\bm{h}_i}k_1\mc{\bm B}\bm\delta_p-\bm\xi^\top\bm{Z}\bm{D}_{\bm{h}_i^\perp}\bm\xi\\
&+(\bm{1}_n\otimes\bm{v}_c)^\top(\bm{I}_{dn}-\bm{Z}+\bm{Z}\bm{D}_{\bm{h}_i^\perp}+\bm{Z}\bm{D}_{\bm{h}_i})\mc{\bm B}\bm\delta_p\\
=&-k_1\bm\delta_p^\top\mc{\bm B}\bm{Z}\bm{D}_{\bm{h}_i}\mc{\bm B}\bm\delta_p-\bm\xi^\top\bm{Z}\bm{D}_{\bm{h}_i^\perp}\bm\xi\\
\leq& ~0,\numberthis \label{eq:dotV_disp}
\end{align*}
where, the second equality makes use of $\mc{\bm B}(\bm{1}_n\otimes \bm v_c)=\bm 0$ and $\bm h^\top \dot{\bm{h}}=0$, and in the third equality, we have used $(\mc{\bm B}\bm{Z}\bm{D}_{\bm{h}_i})^\top=\bm{Z}\bm{D}_{\bm{h}_i}\mc{\bm B}^\top=\bm{Z}\bm{D}_{\bm{h}_i}\mc{\bm B}$. It follows from $\dot{V}\leq 0$ that $V(t)\leq V(0)$, and hence $\bm\delta_p(t)$ and $\bm\xi(t)$ are always bounded. Consequently, by LaSalle's invariance principle, $\bm\delta_{p}(t)$ and $\bm\xi(t)$ asymptotically converge to the largest invariant set where $\dot{V}=0$. By \eqref{eq:dotV_disp}, $\dot{V}=0$ leads to
\begin{align}
&\bm{Z}\bm{D}_{\bm{h}_i}\mc{\bm B}\bm\delta_p=\bm{0} \text{~and~} 
 \bm{ZD}_{\bm{h}_i^\perp}\bm\xi=\bm{0}.\label{eq:dot_V_disp_invariant2}
\end{align}
Following a similar argument used in Proof of Thm. \ref{thm:bearing_only_asympt_converg}, we have that the agents' velocities, i.e., $\bm{v}_i=\bm{\xi}_i,i\in \mc{V}_f$, {are constant} and $\bm{Z}\mc{\bm B}\bm\delta_p=\bm{0}$. The latter implies that 
\begin{equation*}
\bm\delta_p^\top \bm{Z}\mc{\bm B}\bm\delta_p=\bm{0}\Rightarrow\bm\delta_{p_f}^\top\mc{\bm B}_{ff}\bm\delta_{p_f}=\bm{0},
\end{equation*} which leads to $\bm p_f=\bm p_f^*$. Thus, $\bm p_f(t)\rightarrow\bm p_f^*(t)$ {globally and} asymptotically as $t\rightarrow \infty$.
\end{proof}

{A sufficient condition for inter-agent collision avoidance can be obtained as follows.} From $\dot{V}\leq 0$, one has that $V(t)\leq \beta$ for all $t\geq 0$ if initially $V(0)\leq \beta$ for a constant $\beta>0$ . It follows that $\frac{1}{2}\bm\delta_p^\top\mc{\bm B}\bm\delta_p\leq \beta \Leftrightarrow \abs{\bm\delta_p}\leq \phi :=\sqrt{2\beta/\lambda_{\min}(\mc{\bm B}_{ff})}, \forall t\geq 0$. Thus, when $\beta$ is sufficiently small such that $\phi \leq \epsilon$, where $\epsilon$ is given in \eqref{eq:epsilon}, then $\abs{\bm{p}_i-\bm{p}_j}\geq \kappa$, $\forall i,j\in \mc{V}, i\neq j,$ and for all $t\geq 0$ and $\bm{p}(t)\rightarrow \bm{p}^*(t)$ asymptotically as $t\rightarrow \infty$ (see Corollary \ref{coroll:no_collision}). In other words, no collision will happen between the agents.

\section{Simulation}\label{sec:simulation}
We provide two simulations of the formation tracking of six nonholonomic agents in the {$3$-D space} under the control laws \eqref{eq:tracking_control_law_i} and \eqref{eq:disp_tracking_control_law_i}, respectively, in Fig. \ref{fig:simulation}. A video of the simulations {and more simulation results for the $2$-D case} can be found in \url{https://youtu.be/EM-cgxof8bk}. The graph $\bar{\mc{G}}$ is defined as $\mc{V}=\{1,\ldots,6\}$, and $\bar{\mc{E}}=\{(1,2),(2,3),(1,3),(1,4),(3,4),(3,5),(3,6),(4,5),(4,6)\}$ depicted as blue solid lines in Fig. \ref{fig:traject_bearing}. Agents $1$ and $2$ are leaders and the four other agents are followers. The leaders' initial positions are $\bm{p}_1(0)=[10,0,{0}]^\top$ and $\bm{p}_2(0)=[10,5,{0}]^\top$ (m), and their common heading and velocity are $\bm{h}_c=[\cos(\pi/6),\sin(\pi/6),{0}]^\top$ and $u_c=0.15$ (m/s), respectively. {Each agent maintains a body-fixed coordinate frame shown in red-green-blue, whose first (red) axis points to the heading direction. The unit vectors, i.e., $\bm e\in \mb{R}^3$, along the coordinate axes of agent $i$ thus obey $\dot{\bm e}=\bm \omega_i\times\bm e$ as in \eqref{eq:unicycle_model}.} The desired bearing vectors are $\bm{g}_{12}^*=\bm{g}_{43}^*=\bm{g}_{56}^*=[0,1,{0}]^\top,$ $\bm{g}_{14}^*=\bm{g}_{23}^*=[-1,0,{0}]^\top$, {$\bm{g}_{36}^*=\bm{g}_{45}^*=[-5/41,0,4/41]^\top$}, $\bm{g}_{13}^*=[-1/\sqrt{2},1/\sqrt{2},{0}]^\top$ and $\bm{g}_{35}^*={[-5/66,-5/66,2/33]^\top}$. Note that the desired formation $(\bar{\mc{G}},\bm{p}^*)$ is infinitesimally bearing rigid.
\subsection{Bearing-only formation tracking control}
Simulation results of the formation tracking of the system under the tracking control law \eqref{eq:tracking_control_law_i} are given in Figs. \ref{fig:simulation} (a, c, e).
The controller's gains are chosen as $k_1=15$, $k_2=7$. It can be seen that the system achieves the desired formation and the followers' velocities $\bm{v}_i$ converge to $\bm{v}_c=u_c\bm{h}_c$. Therefore, the followers' headings $\bm{h}_i$ align with $\bm{h}_c$ asymptotically. We observe from the simulation that selecting higher control gains $k_1$ and $k_2$ would result in faster convergence speed but higher velocity inputs in the transient state.
\begin{figure*}[t]
\centering
\begin{subfigure}[b]{0.48\textwidth}
\centering
\includegraphics[width=0.9\textwidth]{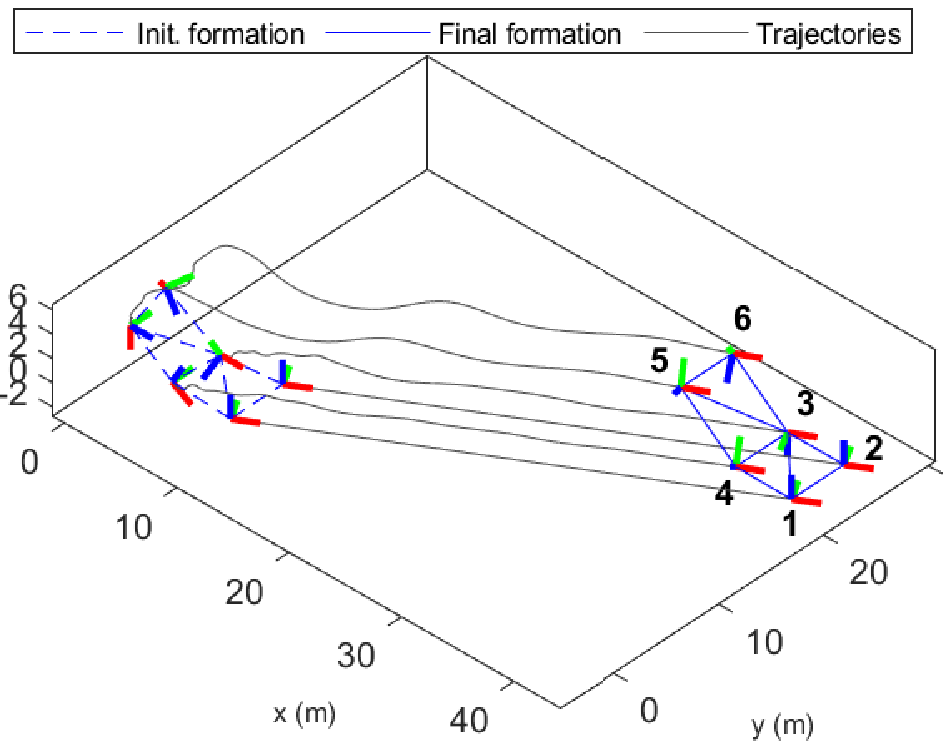}
\caption{Trajectories of the agents.}
\label{fig:traject_bearing}
\end{subfigure}
\begin{subfigure}[b]{0.48\textwidth}
\centering
\includegraphics[width=0.9\textwidth]{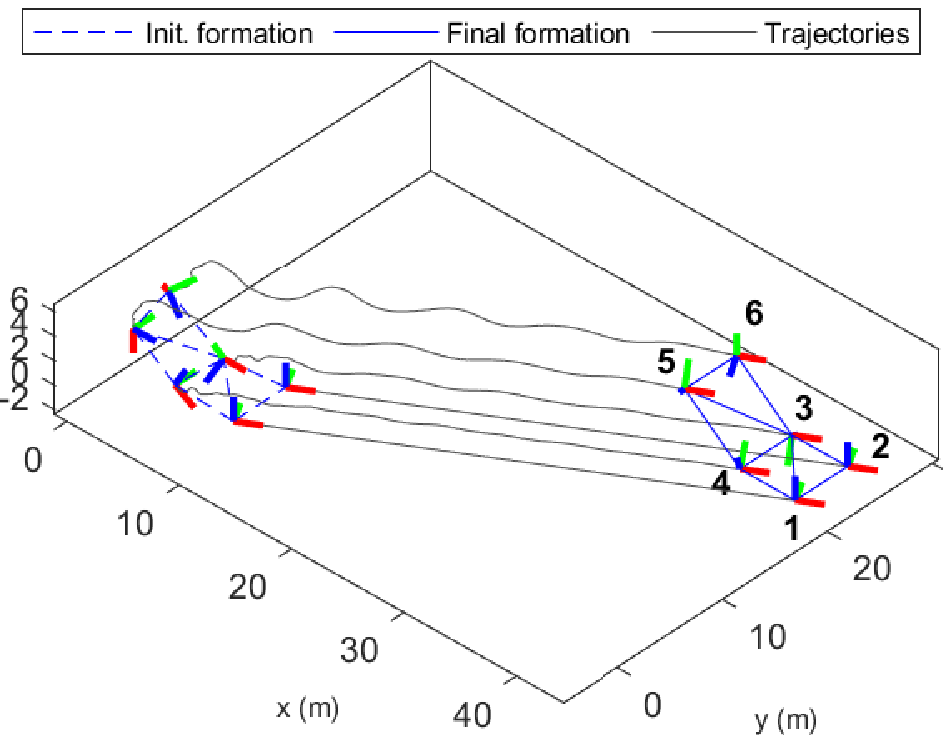}
\caption{Trajectories of the agents.}
\label{fig:traject_disp}
\end{subfigure}
\begin{subfigure}[b]{0.48\textwidth}
\centering
\includegraphics[width=0.9\textwidth]{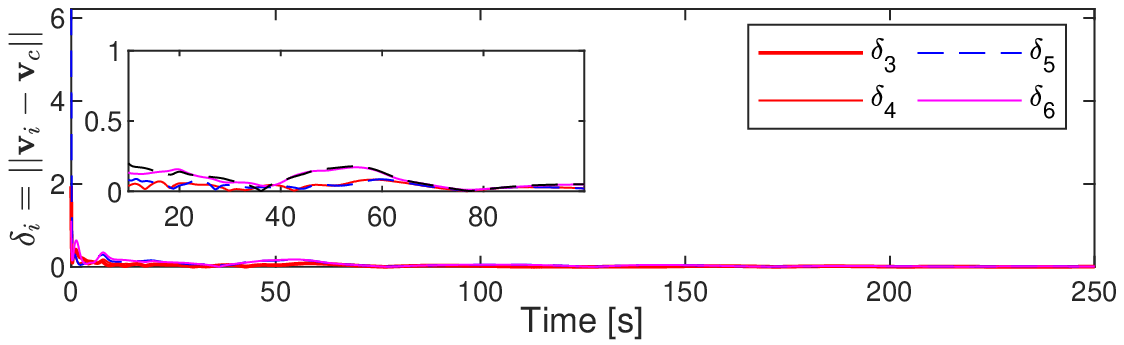}
\caption{Velocity tracking errors $\abs{\bm{v}_i-\bm{v}_c},i\in \mc{V}_f$.}
\label{fig:vel_bearing}
\end{subfigure}
\begin{subfigure}[b]{0.48\textwidth}
\centering
\includegraphics[width=0.9\textwidth]{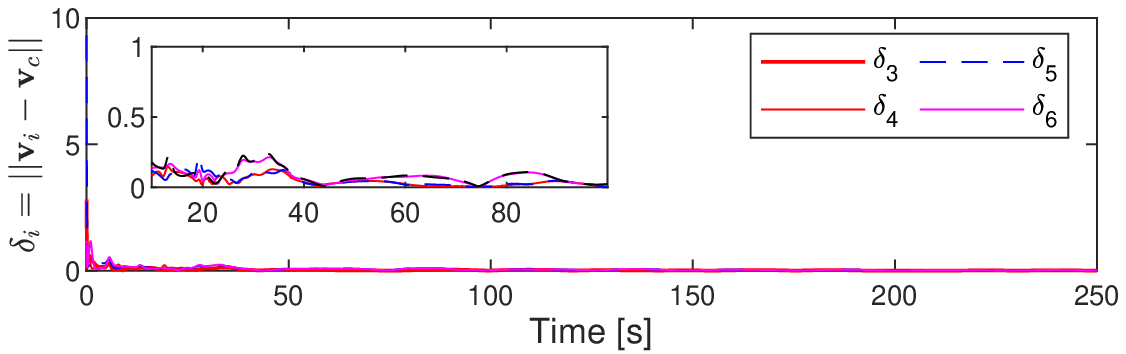}
\caption{Velocity tracking errors $\abs{\bm{v}_i-\bm{v}_c},i\in \mc{V}_f$.}
\label{fig:vel_disp}
\end{subfigure}
\begin{subfigure}[b]{0.48\textwidth}
\centering
\includegraphics[width=0.9\textwidth]{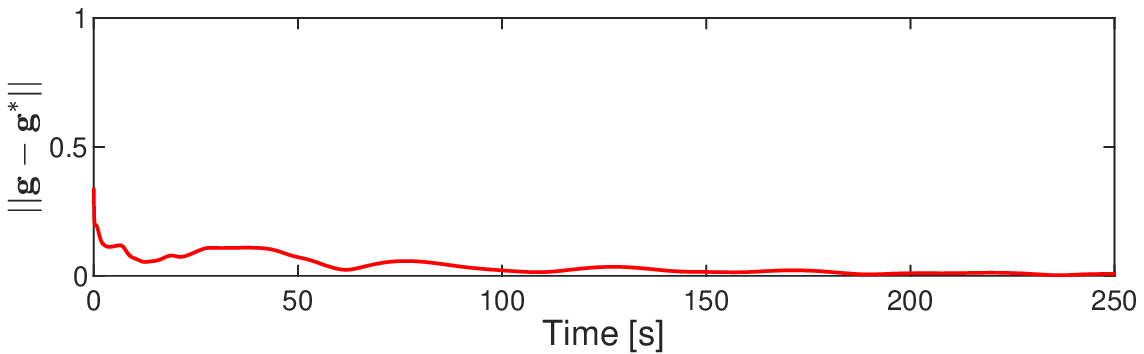}
\caption{Bearing error $\abs{\bm g - \bm {g}^*}$.}
\label{fig:bearing_err_bearing}
\end{subfigure}
\begin{subfigure}[b]{0.48\textwidth}
\centering
\includegraphics[width=.9\textwidth]{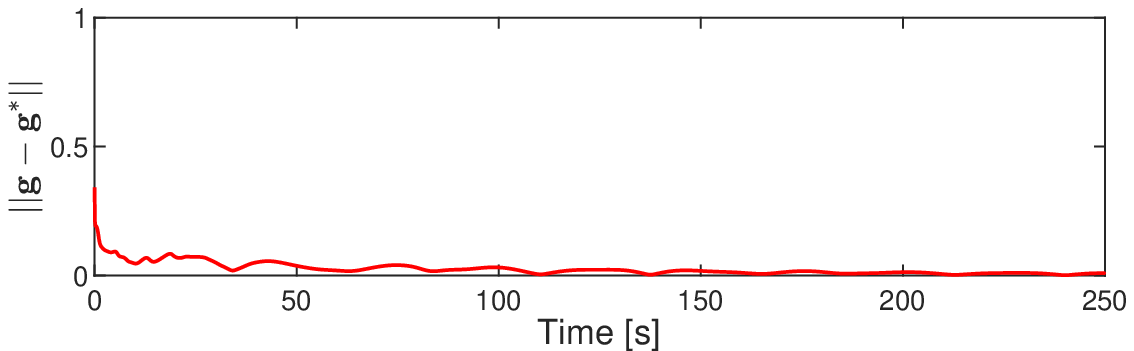}
\caption{Bearing error $\abs{\bm g - \bm {g}^*}$.}
\label{fig:bearing_err_disp}
\end{subfigure}
\caption{Formation tracking control of $6$ agents {in $\mb{R}^3$}: (a, c, e) under control law \eqref{eq:tracking_control_law_i} and (b, d, f) under control law \eqref{eq:disp_tracking_control_law_i}.}
\label{fig:simulation}
\end{figure*}

\subsection{Bearing-based formation tracking control}
Simulation results of the formation tracking of the system under the control law \eqref{eq:disp_tracking_control_law_i} are shown in Figs. \ref{fig:simulation} (b, d, f). 
In the simulation, the control gains are chosen as $k_1=5$, $k_2=3$. As can be seen, $\bm{p}\rightarrow \bm{p}^*$ and $\bm{v}_i\rightarrow \bm{v}_c$ asymptotically for all $i\in \mc{V}_f$.  The convergence of the system is slightly faster than that of the first simulation in the preceding subsection even with much lower control gains. This is due to the fact that the magnitude of the reference control vector $\bm{r}_i$ in \eqref{eq:disp_tracking_control_law_i} increases proportionally with the inter-agent distances. In contrast, $\bm{r}_i$ in \eqref{eq:tracking_control_law_i} is always bounded by $\abs{\sum_{j\in \mc{N}_i}(\bm{g}_{ij}-\bm{g}_{ij}^*)}\leq 2|\mc{N}_i|$, where $|\mc{N}_i|$ is the cardinality of $\mc{N}_i$.
\subsection{Bearing-constrained formation tracking of unicycles}
\begin{figure*}[t]
\centering
\begin{subfigure}[b]{0.47\textwidth}
\centering
\includegraphics[width=0.9\textwidth]{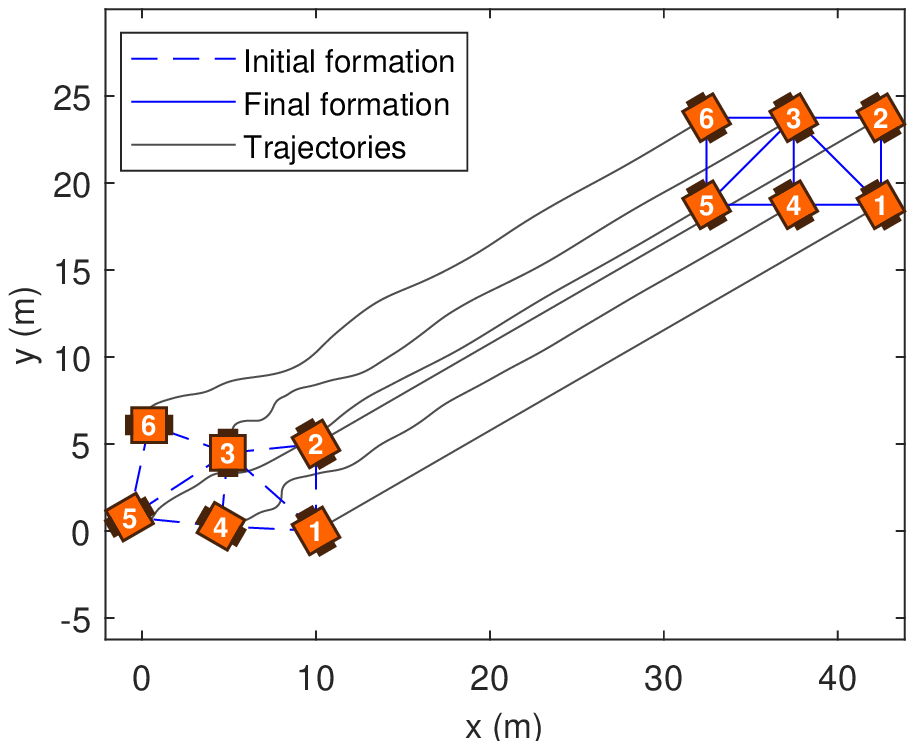}
\caption{Trajectories of the agents.}
\label{fig:traject_bearing_2D}
\end{subfigure}
\begin{subfigure}[b]{0.47\textwidth}
\centering
\includegraphics[width=0.9\textwidth]{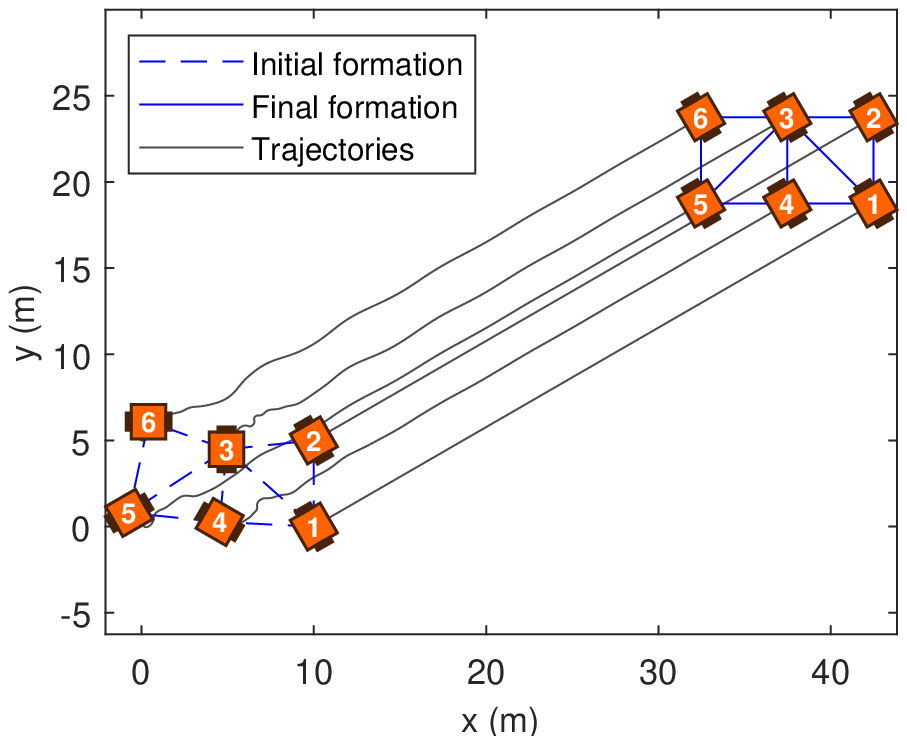}
\caption{Trajectories of the agents.}
\label{fig:traject_disp_2D}
\end{subfigure}
\begin{subfigure}[b]{0.47\textwidth}
\centering
\includegraphics[width=0.88\textwidth]{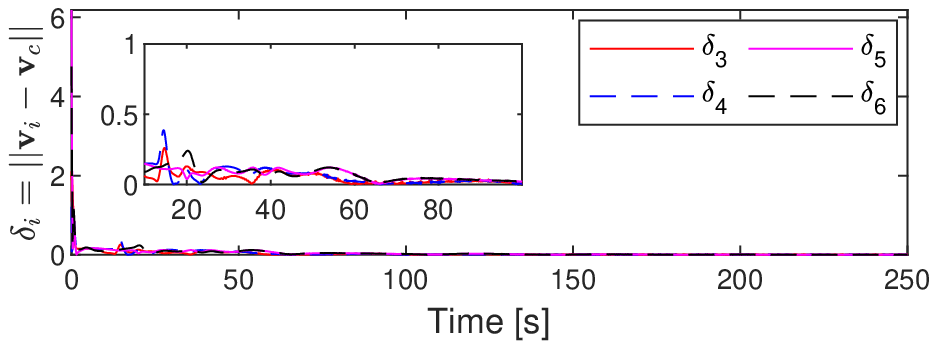}
\caption{Velocity tracking errors $\abs{\bm{v}_i-\bm{v}_c},i\in \mc{V}_f$.}
\label{fig:vel_bearing_2D}
\end{subfigure}
\begin{subfigure}[b]{0.47\textwidth}
\centering
\includegraphics[width=0.88\textwidth]{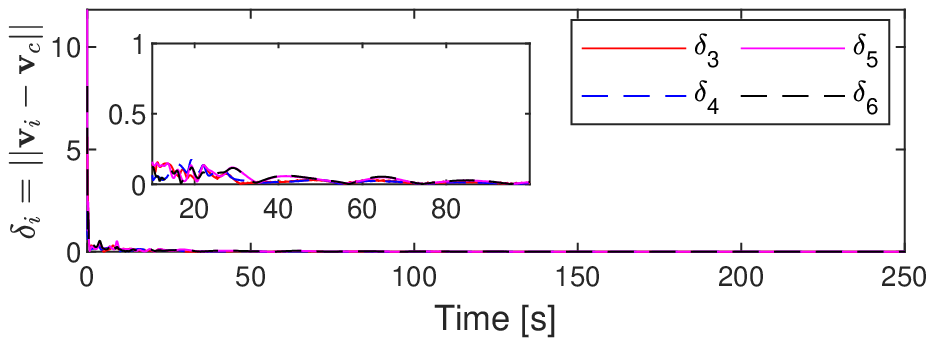}
\caption{Velocity tracking errors $\abs{\bm{v}_i-\bm{v}_c},i\in \mc{V}_f$.}
\label{fig:vel_disp_2D}
\end{subfigure}
\begin{subfigure}[b]{0.47\textwidth}
\centering
\includegraphics[width=0.88\textwidth]{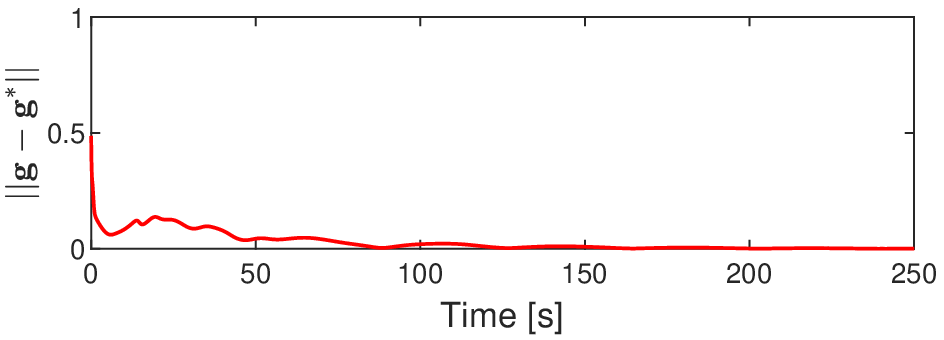}
\caption{Bearing error $\abs{\bm g - \bm {g}^*}$.}
\label{fig:bearing_err_bearing_2D}
\end{subfigure}
\begin{subfigure}[b]{0.47\textwidth}
\centering
\includegraphics[width=.88\textwidth]{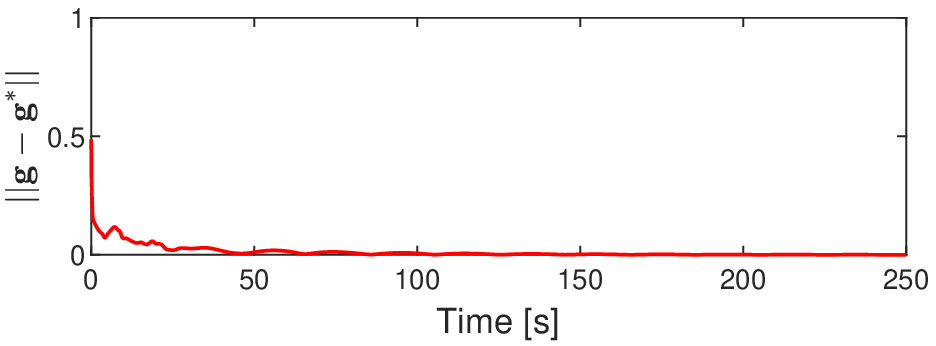}
\caption{Bearing error $\abs{\bm g - \bm {g}^*}$.}
\label{fig:bearing_err_disp_2D}
\end{subfigure}
\caption{Formation tracking control of $6$ unicycle agents in $2$-D: (a, c, e) under control law \eqref{eq:tracking_control_law_i} and (b, d, f) under control law \eqref{eq:disp_tracking_control_law_i}.}
\label{fig:simulation_2D}
\end{figure*}
This subsections provides two simulations of the formation tracking of six nonholonomic agents in the $2$-D plane under the control laws \eqref{eq:tracking_control_law_i} and \eqref{eq:disp_tracking_control_law_i}, respectively, in Fig. \ref{fig:simulation_2D}. A video of the simulations can also be seen in \url{https://youtu.be/m2gB-mOcVjk}.
The graph $\bar{\mc{G}}$ is defined as $\mc{V}=\{1,\ldots,6\}$, and $\bar{\mc{E}}=\{(1,2),(2,3),(1,3),(1,4),(3,4),(3,5),(3,6),(4,5),(4,6)\}$ depicted as blue solid lines in Fig. \ref{fig:traject_bearing_2D}. Agents $1$ and $2$ are leaders and the four other agents are followers. The leaders' initial positions are $\bm{p}_1(0)=[10,0]^\top$ and $\bm{p}_2(0)=[10,5]^\top$ (m), and their common heading and velocity are $\bm{h}_c=[\cos(\pi/6),\sin(\pi/6)^\top$ and $u_c=0.2$ (m/s), respectively. The desired bearing vectors are $\bm{g}_{12}^*=\bm{g}_{43}^*=\bm{g}_{56}^*=[0,1]^\top,$ $\bm{g}_{14}^*=\bm{g}_{23}^*=\bm{g}_{36}^*=\bm{g}_{45}^*=[-1,0]^\top$, $\bm{g}_{13}^*=[-1/\sqrt{2},1/\sqrt{2}]^\top$ and $\bm{g}_{35}^*=[-1/\sqrt{2},-1/\sqrt{2}]^\top$. The control gains are selected the same as before.

It can be observed that, under both the control laws, the system achieves the desired formation and the followers' velocities $\bm{v}_i$ converge to $\bm{v}_c=u_c\bm{h}_c$ eventually. Therefore, the followers' headings $\bm{h}_i$ align with $\bm{h}_c$ asymptotically. 
\section{Conclusion}\label{sec:conclusion}
In this letter, we proposed two novel adaptive formation tracking control schemes for nonholonomic agents with constant velocity leaders based solely on the inter-agent bearings and displacements, respectively. In the proposed control laws, the followers do not know the constant reference velocity nor need to communicate variables with their neighbors. For future works, bearing-only formation tracking control with time-varying leaders' velocity or for underactuated autonomous surface vehicles are worth investigating.

\nocite{} 
\bibliographystyle{IEEEtran}
\bibliography{IEEEabrv,quoc2018,quoc2019} 

\begin{thebibliography}{10}
\providecommand{\url}[1]{#1}
\csname url@samestyle\endcsname
\providecommand{\newblock}{\relax}
\providecommand{\bibinfo}[2]{#2}
\providecommand{\BIBentrySTDinterwordspacing}{\spaceskip=0pt\relax}
\providecommand{\BIBentryALTinterwordstretchfactor}{4}
\providecommand{\BIBentryALTinterwordspacing}{\spaceskip=\fontdimen2\font plus
\BIBentryALTinterwordstretchfactor\fontdimen3\font minus
  \fontdimen4\font\relax}
\providecommand{\BIBforeignlanguage}[2]{{%
\expandafter\ifx\csname l@#1\endcsname\relax
\typeout{** WARNING: IEEEtran.bst: No hyphenation pattern has been}%
\typeout{** loaded for the language `#1'. Using the pattern for}%
\typeout{** the default language instead.}%
\else
\language=\csname l@#1\endcsname
\fi
#2}}
\providecommand{\BIBdecl}{\relax}
\BIBdecl

\bibitem{Zhao2016tac}
S.~Zhao and D.~Zelazo, ``Bearing rigidity and almost global bearing-only
  formation stabilization,'' \emph{IEEE Trans. Autom. Control}, vol.~61, no.~5,
  pp. 1255--1268, 2015.

\bibitem{Quoc2018ccta}
Q.~V. {Tran}, S.~H. {Park}, and H.-S. {Ahn}, ``Bearing-based formation control
  via distributed position estimation,'' in \emph{2018 IEEE Conference on
  Control Technology and Applications (CCTA)}, 2018, pp. 658--663.

\bibitem{Schiano2016}
F.~Schiano, A.~Franchi, D.~Zelazo, and P.~R. Giordano, ``A rigidity-based
  decentralized bearing formation controller for groups of quadrotor {UAV}s,''
  in \emph{Proc. the 2016 IEEE/RSJ Int. Confer. Intelligent Robots and Systems
  (IROS)}, 2016, pp. 5099--5106.

\bibitem{Quoc2018tcns}
Q.~V. Tran, M.~H. Trinh, D.~Zelazo, D.~Mukherjee, and H.-S. Ahn, ``Finite-time
  bearing-only formation control via distributed global orientation
  estimation,'' \emph{IEEE Trans. Control Network Syst.}, vol.~2, no.~6, pp.
  702--712, 2019.

\bibitem{Xli2020TCyber}
X.~{Li}, C.~{Wen}, and C.~{Chen}, ``Adaptive formation control of networked
  robotic systems with bearing-only measurements,'' \emph{IEEE Trans. Cybern.},
  vol.~51, no.~1, pp. 199--209, 2021.

\bibitem{Bishop2011}
A.~N. Bishop, ``Distributed bearing-only formation control with four agents and
  a weak control law,'' in \emph{Proc. IEEE Int. Confer. Control \& Automation
  (IEEE ICCA'11)}, 2011, pp. 30--35.

\bibitem{Zhao2019tac}
S.~Zhao, Z.~Li, and Z.~Ding, ``Bearing-only formation tracking control of
  multiagent systems,'' \emph{IEEE Trans. Autom. Control}, vol.~64, no.~11, pp.
  4541--4554, 2019.

\bibitem{Tron2016csm}
R.~Tron, J.~Thomas, G.~Loianno, K.~Daniilidis, and V.~Kumar, ``A distributed
  optimization framework for localization and formation control: applications
  to vision-based measurements,'' \emph{IEEE Control Syst. Mag.}, vol.~36,
  no.~4, pp. 22--44, 2016.

\bibitem{XPeng2020Tmecha}
X.~Peng, Z.~Sun, K.~Guo, and Z.~Geng, ``Mobile formation coordination and
  tracking control for multiple nonholonomic vehicles,'' \emph{IEEE/ASME Trans.
  Mechatron.}, vol.~25, no.~3, pp. 1231--1242, 2020.

\bibitem{Qingkai2018SCL}
Q.~Yang, M.~Cao, H.~G. de~Marina, H.~Fang, and J.~Chen, ``Distributed formation
  tracking using local coordinate systems,'' \emph{Syst. Control Lett.}, vol.
  111, pp. 70--78, 2018.

\bibitem{Khaledyan2018acc}
M.~Khaledyan and M.~de~Queiroz, ``Translational maneuvering control of
  nonholonomic kinematic formations: Theory and experiments,'' in \emph{Proc.
  Amer. Control Confer. (ACC)}, 2018, pp. 2910--2915.

\bibitem{Zhao2017}
S.~Zhao and D.~Zelazo, ``Translational and scaling formation maneuver control
  via bearing-based approach,'' \emph{IEEE Trans. Control Network Syst.},
  vol.~4, no.~3, pp. 429--438, 2017.

\bibitem{YHuang2021tcns}
Y.~Huang and Z.~Meng, ``Bearing-based distributed formation control of multiple
  vertical take-off and landing uavs,'' \emph{IEEE Trans. Control Network
  Syst.}, vol.~8, no.~3, pp. 1281--1292, 2021.

\bibitem{Minh2021auto}
M.~H. Trinh, Q.~V. Tran, D.~V. Vu, P.~D. Nguyen, and H.-S. Ahn, ``Robust
  tracking control of bearing-constrained leader–follower formation,''
  \emph{Automatica}, vol. 131, p. 109733, 2021.

\bibitem{JZhaoLCSS2021}
J.~Zhao, X.~Yu, X.~Li, and H.~Wang, ``Bearing-only formation tracking control
  of multi-agent systems with local reference frames and constant-velocity
  leaders,'' \emph{IEEE Control Syst. Lett.}, vol.~5, no.~1, pp. 1--6, 2021.

\bibitem{Minh2022LCSS}
M.~H. Trinh and H.-S. Ahn, ``Finite-time bearing-based maneuver of acyclic
  leader-follower formations,'' \emph{IEEE Control Syst. Lett.}, vol.~6, pp.
  1004--1009, 2022.

\bibitem{XLi2021Tcyb}
X.~Li, C.~Wen, X.~Fang, and J.~Wang, ``Adaptive bearing-only formation tracking
  control for nonholonomic multiagent systems,'' \emph{IEEE Trans. Cybern.},
  pp. 1--11, 2021.

\bibitem{Quoc2020tcns}
Q.~V. Tran and H.-S. Ahn, ``Distributed formation control of mobile agents via
  global orientation estimation,'' \emph{IEEE Trans. Control Network Syst.},
  vol.~4, no.~7, pp. 1654--1664, 2020.

\bibitem{Xiaodong2022AutoSinica}
X.~He, Z.~Sun, Z.~Geng, and A.~Robertsson, ``Exponential set-point
  stabilization of underactuated vehicles moving in three-dimensional space,''
  \emph{IEEE/CAA J. Autom. Sin.}, vol.~9, no.~2, pp. 270--282, 2022.

\bibitem{Mesbahi2010}
M.~Mesbahi and M.~Egerstedt, \emph{Graph Theoretic Methods in Multiagent
  Networks}.\hskip 1em plus 0.5em minus 0.4em\relax Princeton University Press,
  2010.

\bibitem{Ma2004}
Y.~Ma, S.~Soatto, J.~Kosecka, and S.~Sastry, \emph{An Invitation to 3D
  Vision}.\hskip 1em plus 0.5em minus 0.4em\relax New York: Springer, 2004.

\bibitem{Khalil2002}
H.~K. Khalil, \emph{Nonlinear Systems}, 3rd~ed.\hskip 1em plus 0.5em minus
  0.4em\relax Prentice Hall, 2002.

\end{thebibliography}

%
%
\end{document}